\documentclass[12pt]{article}

\usepackage{a4wide}
\usepackage{amssymb,amsmath,array,amsthm, amsfonts, amscd}
\usepackage{colortbl}
\usepackage{algpseudocode}
\usepackage{algorithm}
\usepackage{tikz}
\usepackage{graphicx}
\usepackage{comment}
\allowdisplaybreaks[1]

\input xy
\xyoption{matrix}\xyoption{arrow}
\def\edge{\ar@{-}}
\def\dedge{\ar@{.}}

\long\def\ignore#1{#1}

\newtheorem{theorem}{Theorem}[section]
\newtheorem{corollary}[theorem]{Corollary}
\newtheorem{lemma}[theorem]{Lemma}
\newtheorem{proposition}[theorem]{Proposition}

\theoremstyle{definition}
\newtheorem{definition}[theorem]{Definition}
\newtheorem{claim}[theorem]{Claim}
\newtheorem{example}[theorem]{Example}
\newtheorem{remark}[theorem]{Remark}

\def\ca{{\mathcal A}}

\def\ch{{\mathcal H}}

\def\co{{\mathcal O}}

\def\mn{{\mathbb N}}
\def\mq{{\mathbb Q}}

\def\mz{{\mathbb Z}}

\def\k{{K}}

\def\oq{\mathcal{O}_{q}}

\def\oqmmnk{\oq(M(m,n))}
\def\oqmnnk{\oq(M(n,n))}
\def\oqmkn{\oq(M(k,n))}
\def\oqmknk{\oq(M(k,n))}
\def\oqmkp{\oq(M(k,p))}

\def\oqgkn{\oq(G(k,n))}

\def\oqgtwofour{\oq(G(2,4))}

\def\oqgthreesix{\oq(G(3,6))}

\title{The automorphism group of the quantum grassmannian
\footnote{This research was partly supported by EPSRC grant EP/R009279/1.}
}

\author{S Launois and T H Lenagan}
\date{}
\begin{document}
\maketitle

\begin{abstract}
We calculate the automorphism group of the generic quantum grassmannian.
\end{abstract} 


\section{Introduction} \label{section-introduction} 

The quantum grassmannian $\oqgkn$ is a noncommutative algebra that is a deformation of the homogeneous coordinate ring of the classical grassmannian of $k$-planes in $n$-space. In this paper, we calculate the automorphism group of the quantum grassmannian in the generic case where the deformation parameter $q$ is not a root of unity.

Typically, quantised coordinate algebras are much more rigid than their classical counterparts, in the sense that the automorphism group of the quantum object is much smaller than that of the classical object. Nevertheless, it has proven difficult to calculate these automorphism groups and only a few examples are known where the calculation has been completed, see, for example, \cite{ac,gy,launois,ll1,ll2,y1,y2}. 
The automorphism group of quantum matrices \cite{ll1,y1} will prove crucial in our present work.

The quantum grassmannian $\oqgkn$ is generated as an algebra by the $k\times k$ quantum minors of the quantum matrix algebra $\oqmknk$. These generators are called quantum Pl\"ucker coordinates and there is a natural partial order on the quantum Pl\"ucker coordinates  which is illustrated in the case of $\oqgthreesix$ in Figure~\ref{figure-3x6-ordering}. There are two obvious sources of automorphisms for $\oqgkn$. The first is by restricting column  automorphisms of $\oqmknk$ to the subalgebra $\oqgkn$, these automorphisms are described in Section~\ref{section-dhom-autos}. The second is via studying certain automorphisms of the (noncommutative) dehomogenisation of $\oqgkn$ which is isomorphic to a skew Laurent extension of $\oqmkp$, with $p=n-k$,  as we shall see in Section~\ref{section-dehomogenisation}. 

In Section~\ref{section-dhom-autos} we study these ``obvious'' automorphisms of $\oqgkn$ and consider the relations between them. We then claim that these provide all of the automorphisms of $\oqgkn$, and justify the claim in the following sections. 

The quantum grassmannian carries the structure of an $\mn$-graded algebra, with each quantum Plucker coordinate having degree one. In Section~\ref{section-adjusting-autos}, we exploit this grading in a series of lemmas to see that we can essentially fix the minimal and maximal elements in the poset after allowing adjustment by the automorphisms that we have found in  Section~\ref{section-dhom-autos}. 

In Section~\ref{section-transfer} we study these adjusted automorphisms and show that such an automorphism induces, via the dehomogenisation equality, an automorphism of $\oqmknk$. Once this has been done, our main result follows easily in Section~\ref{section-main-result} from the known structure of the automorphism group of quantum matrices. 


\section{Basic definitions} \label{section-basic-definitions}

Throughout the paper, we work with a field $\k$ and a nonzero element $q\in\k$ which is not a root of unity.

The algebra of $m\times n$ quantum matrices over $\k$, denoted by $\oqmmnk$, is 
the algebra generated over $\k$ by 
$mn$ indeterminates 
$x_{ij}$, with $1 \le i \le m$ and $1 \le j \le n$,  which commute with the elements of 
$\k$ and are subject to the relations:
\[
\begin{array}{ll}  
x_{ij}x_{il}=qx_{il}x_{ij},&\mbox{ for }1\le i \le m,\mbox{ and }1\le j<l\le
n\: ;\\ 
x_{ij}x_{kj}=qx_{kj}x_{ij}, & \mbox{ for }1\le i<k \le m, \mbox{ and }
1\le j \le n \: ; \\ 
x_{ij}x_{kl}=x_{kl}x_{ij}, & \mbox{ for } 1\le k<i \le m,
\mbox{ and } 1\le j<l \le n \: ; \\
x_{ij}x_{kl}-x_{kl}x_{ij}=(q-q^{-1})x_{il}x_{kj}, & \mbox{ for } 1\le i<k \le
m, \mbox{ and } 1\le j<l \le n.
\end{array}
\]
It is well known that $\oqmmnk$ is an iterated Ore extension over $\k$ with the $x_{ij}$ added 
in lexicographic order. An immediate consequence is that $\oqmmnk$ is a noetherian domain. 

When $m=n$, the {\em quantum determinant} $D_q$ is defined by;
\[
D_q:= \sum\,(-q)^{l(\sigma)}x_{1\sigma(1)}\dots x_{n\sigma(n)},
\]
where the sum is over all permutations $\sigma$ of $\{1,\dots,n\}$. 

The quantum determinant is a central element in the algebra of quantum matrices $\oqmnnk$. 

If $I$ and $J$ are $t$-element subsets of $\{1,\dots,m\}$ and $\{1,\dots,n\}$, respectively, 
then the {\em quantum minor} $[I\mid J]$ is defined to be the quantum determinant of the 
$t\times t$ quantum matrix subalgebra generated by the variables $x_{ij}$ with $i\in I$ and $j\in J$. 

The {\em homogeneous coordinate ring of the $k\times n$ quantum grassmannian}, $\oqgkn$ (informally known as the {\em quantum grassmannian}) is the subalgebra of $\oqmknk$ generated by the $k\times k$ quantum minors of $\oqmknk$, see, for example, \cite{klr}. 

The quantum grassmannian $\oq(G(1,n))$ is a quantum affine space, and, as such, its 
automorphism group is known, see \cite{ac}; so we will assume throughout this paper that $k>1$. 
Also, we will see in Proposition~\ref{proposition-k-(n-k)} that 
$\oq(G(k,n))\cong \oq(G(n-k,n))$, so in calculating the automorphism group we will assume that $2k\leq n$ (and so $n\geq 4$, as $k\geq 2$).

A $k\times k$ quantum minor of $\oqmknk$ must use all of the $k$ rows, and so we can specify the quantum minor by specifying the columns that define it. With this in mind, we will write 
$[J]$ for the quantum minor $[1,\dots,k\mid J]$, for any $k$-element subset $J$ of $\{1,\dots,n\}$. Quantum minors of this type are called {\em quantum Pl\"ucker coordinates}. The set of quantum Pl\"ucker 
coordinates in $\oqgkn$ is denoted by $\Pi$. There is a natural partial order on $\Pi$ defined in the following way: if $I=[i_1<\dots<i_k]$ and $J=[j_1 < \dots < j_k]$ then $[I]<[J]$ if and only if 
$i_l\leq j_l$ for each $l=1,\dots,k$. This partial order is illustrated for the case of $\oqgthreesix$ in 
Figure~\ref{figure-3x6-ordering}. A  {\em standard monomial} in the quantum Pl\"ucker coordinates is an expression of the form $[I_1][I_2]\dots[I_t]$ where $I_1\leq I_2\leq\dots\leq I_t$ in this partial order. The set of all standard monomials forms a vector space basis of $\oqgkn$ over $\k$, see, for example, \cite[Proposition 2.8]{klr}.

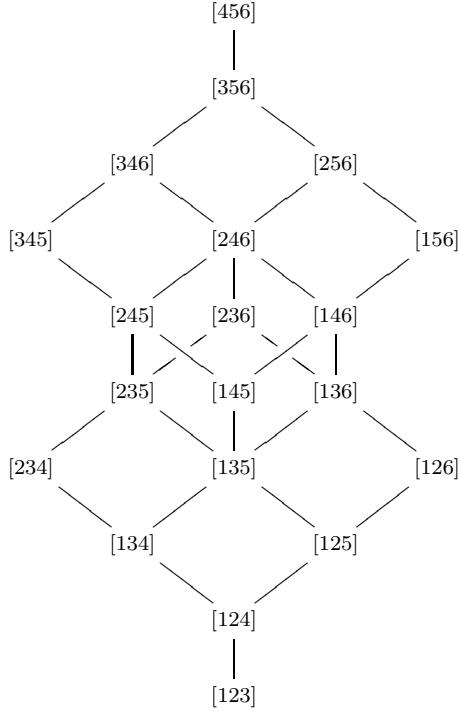
\begin{figure}[ht]
\ignore{
$$\xymatrixrowsep{2.4pc}\xymatrixcolsep{3.2pc}\def\objectstyle{\scriptstyle}
\xymatrix@!0{
 &&  [456 ]\edge[d]\\
 &&   [356 ] \edge[dl] \edge[dr]\\
 &   [346 ] \edge[dl] \edge[dr] &
 &  [256 ] \edge[dl] \edge[dr]\\
  [345 ] \edge[dr] &&   [246 ] \edge[dl] \edge[dr] \edge[d] 
&&   [156 ] 
  \edge[dl]\\
 &   [245 ] \edge[d] \edge[dr] 
 &   [236 ] \edge[dl]|\hole \edge[dr]|\hole &  [146 ]
\edge[dl] \edge[d]\\
 &  [235 ] \edge[dl] \edge[dr] 
 &  [145 ] \edge[d] &  [136 ] \edge[dl] \edge[dr]\\
  [234 ] \edge[dr] && [135 ] \edge[dl] \edge[dr] 
&& [126 ]\edge[dl]\\
 & [134 ] \edge[dr] && [125 ] \edge[dl]\\
 && [124 ] \edge[d]\\
 && [123 ]
}$$
\caption{The partial order on $\Pi$ for $\oq(G(3,6))$.}
\label{figure-3x6-ordering}
}
\end{figure}


\section{Dehomogenisation of $\oqgkn$}\label{section-dehomogenisation}

An element $a$ in a ring $R$ is said to be a {\em normal} element of $R$ provided that $aR=Ra$. If $R$ is a domain then a nonzero normal element $a$ may be inverted, as the Ore condition for the set $S:=\{a^n\}$ is easily verified. Standard results for noncommutative noetherian rings can be found in the books by Goodearl and Warfield \cite{gw} and McConnell and Robson \cite{mcr}.

 Set $u=\{1,\dots, k\}$. Then $[u]$ commutes with all other quantum Pl\"ucker coordinates up to a power of $q$ as the following lemma shows.

\begin{lemma}\label{lemma-how-u-commutes}
Let $[I]$ be a quantum Pl\"ucker coordinate in $\oqgkn$ and
set $d(I):= \# \left(I\backslash (I\cap u)\right)\geq 0$. Then $[u][I]=q^{d(I)}[I][u]$. 
\end{lemma}

\begin{proof} 
This can be obtained from \cite{kl} by combining  Lemma 3.7  and Theorem 3.4 of that paper. (Note that \cite{kl} uses $q^{-1}$ as the deformation parameter where we use $q$; so care must be taken in interpreting their results.) It can also be extracted from \cite[Corollary 1.1]{klr} by setting $[I]=[u]$ in the statement of the corollary and noting that the summation on the right hand side of the displayed equation is then empty. 
\end{proof} 

 As $\oqgkn$ is generated by the  quantum Pl\"ucker coordinates it follows from the previous lemma that the element 
 $[u]$ is a normal element and so we may invert $[u]$ to obtain the overring $\oqgkn[[u]^{-1}]$.

For $1\leq i\leq k$ and $1\leq j\leq n-k$, set

\[
x_{ij}:=[1\dots,\widehat{k+1-i},\dots k, j+k][u]^{-1}\in \oqgkn[[u]^{-1}].
\]

The case $a=1$ of \cite[Theorem 2.2]{lr} shows that the elements $x_{ij}$ generate 
an algebra $R$, say, that sits inside $\oqgkn[[u]^{-1}]$ and is isomorphic to $\co_q(M(k,n-k))$. Also, 
\[
\oqgkn[[u]^{-1}]=R[[u],[u]^{-1}] 
\]
(NB. The way we've fixed things, that really is an equality in the above display, rather than just an isomorphism.) In the rest of this note, we will write $R=\co_q(M(k,p))$ where $p:=n-k$ and when we are operating on the right hand side of this equality, we will write $y$ for $[u]$.

As $d([1\dots,\widehat{k+1-i},\dots k, j+k])=1$, it follows from Lemma~\ref{lemma-how-u-commutes} that 
$[u]x_{ij}=qx_{ij}[u]$ and that $yx_{ij}=qx_{ij}y$. \\

The equality above says that 
\begin{eqnarray}
\label{equation-dehomogenisation}
\oqgkn[[u]^{-1}]=
\oqmkp[y,y^{-1}; \sigma]
\end{eqnarray}
where $\sigma$ is the automorphism of $\oqmkp$ such that $\sigma(x_{ij})=qx_{ij}$ for 
each $i=1,\dots,k$ and $j=1,\dots,p$. We will refer to Equation (\ref{equation-dehomogenisation}) as the {\em dehmogenisation equality}.

The next lemma gives the formulae for passing between quantum minors and Pl\"ucker coordinates in the above equality.
\begin{lemma} \label{lemma-to-and-fro}
(i) 
Let $[I\mid J]$ be a quantum minor in $R=\oqmkp$. Then 
\[
[I\mid J] = [\{1\dots k\}\backslash (k+1-I)\sqcup (k+J)][u]^{-1}\in\oqgkn[[u]^{-1}]\,.
\]
(ii) Let $[L]$ be a quantum Pl\"ucker coordinate in $\oqgkn$ and write $L=L_{\leq k}\cap L_{>k}$ 
where $L_{\leq k}=L\cap\{1,\dots,k\}$ and $L_{>k} =L\cap\{k+1,\dots,n\}$. Then 
\[
[L] = [(k+1)-\left(\{1,\dots,k\}\backslash L_{\leq k}\right)\mid  L_{>k}-k]y\in\oqmkp[y,y^{-1}; \sigma] 
\]
\end{lemma} 

\begin{proof} (i) This formula occurs as the case 
$a=1$ of \cite[Proposition 3.3]{lr} which gives the formula for general quantum 
minors of $R=\oqmkp$ in terms of quantum Pl\"ucker coordinates of $\oqgkn$. \\
(ii) Let $[L]$ be a quantum Pl\"ucker coordinate in $\oqgkn$. Set $I =(k+1)-(\{1,\dots,k\}\backslash L_{\leq k})$ and $J= L_{>k} -k$. Note that $| I |=| J|=|L_{>k}|$ and so we can form the quantum minor $[I\mid J]$. Apply (i) to $[I\mid J]$ to see that 
\begin{eqnarray*}
[I\mid J][u]&=&[\{1,\dots,k\}\backslash \{(k+1)-\{(k+1)-(\{1,\dots,k\}\backslash L_{\leq k})\}\sqcup k+( L_{>k} -k)]\\
&=& [ \{1,\dots,k\} \backslash (\{1,\dots,k\}\backslash L_{\leq k})\sqcup L_{>k}]=[L_{\leq k}\sqcup L_{>k}]=[L], 
\end{eqnarray*}
so that (ii) is established. 
\end{proof} 

The following corollary to the above lemma will be useful in later calculations.

\begin{corollary}\label{corollary-belonging-to-L}
Suppose that $[L]$ is a quantum Pl\"ucker coordinate in $\oqgkn$ and that $[L]=[I\mid J][u]$ for some $[I\mid J]\in\oqmkp$. \\
(i) Let $i\in\{1,\dots,k\}$. Then $i\in I$ if and only if $(k+1)-i\not\in L$.\\
(ii) Let $j\in\{1\dots,p\}$. Then $j\in J$ if and only if $j+k\in L$.
\end{corollary} 

We will use dehomogenisation, see Equation (\ref{equation-dehomogenisation}), in the next three sections to transfer the problem of finding automorphisms 
of $\oqgkn$ to that of finding automorphisms of $\oq(M(k,p))$ where the problem has been solved in 
\cite{ll1} and \cite{y1}. Before doing that, we illustrate the usefulness of dehomogenisation by 
the following two results, the second of which identifies an extra automorphism of $\oqgkn$ 
in the case that $n=2k$.

First, note that there is an isomorphism $\tau$ between the quantum matrix algebras $\oq(M(k,n-k))=K(x_{ij})$ and $\oq(M(n-k,k))=K(x'_{ij})$ that sends $x_{ij}$ to $x'_{ji}$ and more generally sends a quantum minor $[I\mid J]$ of  $\oq(M(k,n-k))$ to the quantum minor $[J\mid I]$ of $\oq(M(n-k,k))$. Note that, with a slight abuse of notation, we denote quantum minors of $\oq(M(n-k,k))$ without dashes to differentiate them from quantum minors from $\oq(M(k,n-k))$. (When $k=n-k$ and $x_{ij}=x'_{ij}$ this is automorphism of quantum matrices given by transposition and so we will refer to $\tau$ as the {\em transpose isomorphism}.) This automorphism extends to an isomorphism from 
$\oq(M(k,n-k))[y,y^{-1};\sigma]$ to $\oq(M(n-k,k))[y',y'^{-1};\sigma']$, where $\sigma(x_{ij}) = qx_{ij}$ and 
$\sigma'(x'_{ij}) = qx'_{ij}$.  We also denote this extension by $\tau$. 

\begin{proposition} \label{proposition-k-(n-k)}
$\oqgkn\cong\oq(G(n-k,n))$ via an automorphism that sends the 
quantum Pl\"ucker coordinate $[L]$ of $\oqgkn$ to the quantum Pl\"ucker coordinate 
$[w_0(\,\widehat{L}\,)]$ of $\oq(G(n-k,n))$, where $\widehat{L}$ is the complement of $L$ in $\{1,\dots,n\}$ and 
$w_0$ is the longest element of the symmetric group on $\{1,\dots,n\}$; 
that is, $w_0$ reverses the order of $\{1,\dots,n\}$.
\end{proposition}

\begin{proof}
Recall that we are assuming that $2k\leq n$. There is an isomorphism
\begin{eqnarray*}  
\lefteqn{\oqgkn[[u]^{-1}]= \oq(M(k,n-k))[y,y^{-1}; \sigma] \cong^1}\\
&&
\oq(M(n-k,k)[y',y'^{-1};\sigma']
\cong\oq(G(n-k,n))[[u']^{-1}],
\end{eqnarray*}
where $[u']=[1,\dots ,n-k]\in\oq(G(n-k,n))$ and $\sigma'(x_{ij}')=qx_{ij}'$. (Note that $\cong^1$ is given by applying the transpose isomorphism $\tau$ that sends $[I\mid J]$ to $[J\mid I]$ and sends $y=[u]=[1,\dots,k]$ to $y'=[u']=[1,\dots,n-k]$.)

We need to track the destination of an arbitrary quantum Pl\"ucker coordinate of 
$[L]$ of $\oqgkn$ under this isomorphism, using the formulae that we have developed above for translating between quantum Pl\"ucker coordinates and quantum minors. 

\begin{eqnarray*}
[L]
&=& 
[L_{\leq k}\sqcup L_{>k}] \\
&=& [(k+1)-(\{1,\dots,k\}\backslash L_{\leq k}) \mid L_{>k} -k]y\qquad (\in\oq(M(k,n-k))y)\\
&\overset{1}{\mapsto}& 
[L_{>k} -k \mid (k+1)-(\{1,\dots,k\}\backslash L_{\leq k})]y'\qquad (\in\oq(M(n-k,k))y')\\
&=&
[\{1,\dots,n-k\}\backslash(n+1 - L_{>k})\sqcup (n-k) + ((k+1)-(\{1,\dots,k\}\backslash 
L_{\leq k}))]\\
&=& [\{1,\dots,n-k\}\backslash(n+1 - L_{>k})\sqcup ((n+1)-(\{1,\dots,k\}\backslash 
L_{\leq k}))]\\
&=&
 [\{1,\dots,n-k\}\backslash w_0(L_{>k})\sqcup w_0(\{1,\dots,k\}\backslash L_{\leq k})]\\
 &=&
 [\{1,\dots,n-k\}\backslash w_0(L_{>k})\sqcup \{n-k+1,\dots n\}\backslash 
 w_0(L_{\leq k})]\\
 &=&
 [\{1,\dots,n\}\backslash \{w_0(L_{>k})\sqcup w_0(L_{\leq k})\}]\\
 &=&
  [\{1,\dots,n\}\backslash w_0(L)]\\
  &=&
 [\,\widehat{w_0(L)}\,]= [w_0(\,\widehat{L}\,)].
\end{eqnarray*}

As the quantum Pl\"ucker coordinates  of $\oqgkn$ generate $\oqgkn$ as an algebra, and  their images 
generate $\oq(G(n-k,n))$ as an algebra, 
this calculation shows that the isomorphism displayed at the beginning of the proof restricts to an 
isomorphism between $\oqgkn$ and $\oq(G(n-k,n))$. 
\end{proof}

An immediate corollary of this result is the following.

\begin{corollary}\label{corollary-diagram-auto} 
When $2k=n$, 
there is an automorphism of $\oqgkn$ which sends the quantum Pl\"ucker coordinate 
$[I]$ to $[w_0(\,\widehat{I}\,\,)]$, where $w_0$  is the longest element of the symmetric group on $\{1,\dots,n\}$.
\end{corollary} 

\begin{remark}
The automorphism 
 in the previous corollary will be called the {\em diagram automorphism}. 
In Figure~\ref{figure-3x6-ordering},  which shows the standard poset for $\oq(G(3,6))$, the effect of this automorphism 
on the quantum Pl\"ucker coordinates is seen by reflection of the poset in the vertical. For example, $[126]$ is sent to 
$[w_0(\,\widehat{126}\,)]=[w_0(345)]=[234]$. There is a diagram automorphism for $\oqgkn$ only in the case that $n=2k$. Note that both the diagram automorphism and the transpose automorphism $\tau$ extend to 
$\oq(G(k,2k)([u]^{-1})=\oq(M(k,k))[y,y^{-1};\sigma]$ and they agree on this common overring, so we denote the diagram automorphism by $\tau$ also.
\end{remark} 


\section{Obvious automorphisms of $\oqgkn$}\label{section-dhom-autos}

There are two obvious sources of automorphisms of $\oqgkn$. The first is via the inclusion $\oqgkn\subseteq \oqmkn$.  
The second is by using the 
the dehomogenisation equality introduced in Section~\ref{section-dehomogenisation}: 
\[
\co_q(M(k,p))[y^{\pm 1}; \sigma] = \oqgkn[[u]^{-1}],
\]
where $p=n-k$ and $u=\{1,\dots,k\}$ while  $\sigma$ is the automorphism of $\oqmkp$ such that $\sigma(x_{ij})=qx_{ij}$ for 
each $i=1,\dots,k$ and $j=1,\dots,p$.

In this section, we introduce these automorphisms and consider the connections between them. 

First,  $\oqgkn$ is a subalgebra of $\oqmknk$ by definition. The torus $\ch_0:=(\k^*)^n$ acts by column multiplication on $\oqmknk$ and this induces an action on $\oqgkn$ 
 defined on quantum Pl\"ucker coordinates by 
\[
(\beta_1,\dots,\beta_n)\cdot 
[i_1,\dots,i_k] = \beta_{i_1}\dots\beta_{i_k}[i_1,\dots,i_k].
\]
This is the torus action on $\oqgkn$ that is considered in  papers such as \cite{lln,llr-selecta}.

Secondly, there is an action of the torus $(\k^*)^{k+p}$ on $\oqmkp$ which operates by row and column scaling, so that 
$(\alpha_1,\dots,\alpha_k;\beta_1,\dots,\beta_p)\cdot x_{ij}=\alpha_i\beta_jx_{ij}$. As $n=k+p$, we can extend this 
to an action of the torus $\ch_1:=(\k^*)^{n+1}$ on $\oqmkp[y^{\pm 1}; \sigma]$ by setting 
\[
(\alpha_0;\alpha_1,\dots,\alpha_k;\beta_1,\dots,\beta_p)\cdot x_{ij}=\alpha_i\beta_jx_{ij},\quad 
(\alpha_0;\alpha_1,\dots,\alpha_k;\beta_1,\dots,\beta_p)\cdot y=\alpha_0y.
\]  Set $h=(\alpha_0;\alpha_1,\dots,\alpha_k;\beta_1,\dots,\beta_p)\in\ch_1$. It is easy to check that 
$h\cdot[I\mid J]= \alpha_I\beta_J[I\mid J]$, where $\alpha_I:=\alpha_{i_1}\dots\alpha_{i_k}$ when 
$I=[\alpha_{i_1},\dots,\alpha_{i_k}]$, and  $\beta_J=\beta_{j_1}\dots\beta_{j_p}$ when $J=[\beta_{j_1},\dots,\beta_{j_k}]$.

The dehomogenisation equality induces an action of $\ch_1$ on $\oqgkn[[u]^{-1}]$. For   
$h=(\alpha_0;\alpha_1,\dots,\alpha_k;\beta_1,\dots,\beta_p)\in\ch_1$ we have 
$h\cdot [u]=\alpha_0[u]$. If  $[L]$ is any other quantum Pl\"ucker coordinate then we may write $[L]=[I\mid J][u]$, where $I= (k+1)-\left(\{1,\dots,k\}\backslash L_{\leq k}\right)$ 
and $J=L_{>k}-k$, by Lemma~\ref{lemma-to-and-fro}.  Then 
\[
h\cdot[L]=h\cdot [I\mid J]\times h\cdot [u]=\alpha_I\beta_J[I\mid J]\times \alpha_0[u]=\alpha_0\alpha_I\beta_J[I\mid J][u]=\alpha_0\alpha_I\beta_J[L].
\]

As the quantum Pl\"ucker coordinates generate $\oqgkn$ and are sent to scalar multiples of themselves by each $h\in\ch_1$, such $h$ act as automorphisms of $\oqgkn$.

We now consider connections between the actions of $\ch_0$ and $\ch_1$ on $\oqgkn$.

\begin{lemma} \label{lemma-H0-is-in-H1}
For every automorphism $g\in\ch_0$ acting on  $\oqgkn$ there is an automorphism $f\in\ch_1$ which has the same action on $\oqgkn$. 
\end{lemma}

\begin{proof}
Let $g=(a_1,\dots,a_n)\in\ch_0$. We seek $f=(\alpha_0;\alpha_1,\dots,\alpha_k;\beta_1,\dots,\beta_p)\in\ch_1$ such that the actions of $g$ and $f$ on $\oqgkn$ are the same.
As $g\cdot [u]= a_1\dots a_k [u]$, we may extend $g$ to act on $\oqgkn[[u]^{-1}]$. The dehomogenisation equality then transfers this action to 
$\co_q(M(k,p))[y^{\pm 1}; \sigma]$. We calculate the action of $g$ on the generators $x_{ij}$ and $y^{\pm 1}$ of $\co_q(M(k,p))[y^{\pm 1}; \sigma]$.
As $y$ corresponds to $[u]$ in the dehomogenisation equality, $g\cdot y= a_1\dots a_k y$ and $g\cdot y^{-1} = (a_1\dots a_k)^{-1}y^{-1}$. Now,
$x_{ij}:=[1\dots,\widehat{k+1-i},\dots k, j+k][u]^{-1}$; so 
 \begin{eqnarray*}
 g\cdot x_{ij} &=& g\cdot [1\dots,\widehat{k+1-i},\dots k, j+k]\times g\cdot [u]^{-1}\\
 &=&
 a_1\dots a_ka_{k+1-i}^{-1}a_{j+k} [1\dots,\widehat{k+1-i},\dots k, j+k]\times (a_1\dots a_k)^{-1} [u]^{-1}\\
 &=&
 a_{j+k}a_{(k+1)-i}^{-1} [1\dots,\widehat{k+1-i},\dots k, j+k][u]^{-1}\\
 &=&
  a_{j+k}a_{(k+1)-i}^{-1} x_{ij}
\end{eqnarray*}

We seek an $f\in\ch$ which has the same effect on the generators $x_{ij}$ and $y$. Set 
$f=(a_1\dots a_k;a_k^{-1},\dots,a_1^{-1};a_{k+1},\dots,a_n)\in\ch$.
Then $f\cdot y=a_1\dots a_k y=g\cdot y$. Also, the entry in $f$ multiplying the $i$th row is $a_{(k+1)-i}^{-1}$ and the entry multiplying the $j$th column is $a_{k+j}$ so 
$f\cdot x_{ij}= a_{(k+1)-i}^{-1}a_{k+j}x_{ij}=g\cdot x_{ij}$.  Hence, $f$ and $g$ agree on the generators $x_{ij}$ and $y^{\pm 1}$ of 
$\co_q(M(k,p))[y^{\pm 1}; \sigma]=\oqgkn[[u]^{-1}]$; so the actions of $f$ and $g$ on $\oqgkn$ are the same.\end{proof}

The converse question is more delicate, as the following example shows.

\begin{example} 
Let $K$ be a field in which there is no element $b$ such that $b^2=2$ (eg $\mq$) and consider $\oqgtwofour$ over this field. Let $f=(1;2,1;1,1)\in\ch_1$. Then there is no element $g=(a_1,a_2,a_3,a_4)\in\ch_0$ whose action on 
$\oqgtwofour$ coincides with the action of $f$ on $\oqgtwofour$.  To verify this claim, we use the formulae in Lemma~\ref{lemma-to-and-fro} to see that 
\[
[12]=[u],\;[13]=x_{11}[u],\;[14]=x_{12}[u],\;[23]=x_{21}[u],\;[24]=x_{22}[u],\;[34]=[12\mid 12][u].
\]
These equations lead to the following actions of $f$ on the quantum Pl\"ucker coordinates: 
\[
f\cdot[12]=[12],\;f\cdot[13]=2[13],\;f\cdot[14]=2[14],\;f\cdot[23]=[23],\;f\cdot[24]=[24],\;f\cdot[34]=2[34].
\]
Next, for $g=(a_1,a_2,a_3,a_4)\in\ch_0$ we see that 
\[
g\cdot[12]=a_1a_2[12],\;g\cdot[13]=a_1a_3[13],\;g\cdot[14]=a_1a_4[14],\;
\]
\[g\cdot[23]=a_2a_3[23],\;g\cdot[24]=a_2a_4[24],\;g\cdot[34]=a_3a_4[34];
\]
so for this $g$ to act in the same way as $f$ we require that 
\[
a_1a_2= 1,\;a_1a_3= 2,\;a_1a_4= 2,\;a_2a_3= 1,\;a_2a_4= 1,\;a_3a_4= 2.
\]
Equations 4 and 5 immediately above show that we need $a_3=a_4\;(=b {\rm ~say})$ and then equation 6 gives $b^2=2$, which is not possible for any $b\in K$. 
\end{example} 

In view of this example, it is appropriate to assume that $K$ is algebraically closed in the following result. 

\begin{lemma}
Suppose that $K$ is algebraically closed. Consider $\oqgkn$ and $\oqmkp$ over $K$. For every automorphism $f\in\ch_1$ acting on $\oqgkn$ there is an automorphism $g\in\ch_0$ which has the same action on $\oqgkn$. 
\end{lemma} 

\begin{proof} 
Consider the set $S$ of elements $f=(\alpha_0;\alpha_1,\dots,\alpha_k;\beta_1,\dots,\beta_p)\in\ch_1$ that are equal to $1$ in all positions except one. The set $S$ generates $\ch_1$ so it is enough to show that the action of each member of $S$ can be realised via the action of an element of $\ch_0$. We look at three cases separately. 

The first case we consider is when one of the $\beta_i$ terms is not equal to $1$. Suppose that the element $\beta_j$ in position $k+1+j$ of $f$ is not equal to $1$ but that all other positions of $f$ contain the element $1$. Note that 
$f\cdot[u]=[u]$. Let $[L]$ be any other quantum Pl\"ucker coordinate in $\oqgkn$ with $[L]=[I\mid J][u]$ for some quantum minor $[I\mid J]\in\oqmkp$. Then $f\cdot[L]=(f\cdot[I\mid J])(f\cdot[u])=(f\cdot[I\mid J])[u]$ so that $f\cdot[L]=[L]$ when $j\not\in J$ and  $f\cdot[L]=\beta_j[L]$ when $j\in J$. 

By using Corollary~\ref{corollary-belonging-to-L}(ii)  we see that $f\cdot[L]=[L]$ if  $j+k\not\in L$ and that $f\cdot[L]=\beta_j[L]$ when $j+k\in L$. An element of $g\in\ch_0$ that has the same effect on quantum Pl\"ucker coordinates is the $g=(g_1,\dots,g_n)$ where $g_{j+k}=\beta_j$ while all other $g_i=1$. 

The next case that we consider is $f=(\alpha_0;\alpha_1,\dots,\alpha_k;\beta_1,\dots,\beta_p)\in\ch_1$ where $\alpha_i$ is not equal to $1$ for some $i\in\{1,\dots,k\}$ but all other entries of $f$ are equal to $1$. Again, $f\cdot[u]=[u]$. Let $[L]$ be any other quantum Pl\"ucker coordinate in $\oqgkn$ with $[L]=[I\mid J][u]$ for some quantum minor $[I\mid J]\in\oqmkp$. Then $f\cdot[L]=(f\cdot[I\mid J])(f\cdot[u])=(f\cdot[I\mid J])[u]$ so that $f\cdot[L]=[L]$ when $i\not\in I$ and  $f\cdot[L]=\alpha_i[L]$ when $i\in I$.
By using Corollary~\ref{corollary-belonging-to-L}(i)  we see that $f\cdot[L]=[L]$ if $(k+1)-i\in L$ and 
$f\cdot[L]=\alpha_i[L]$ when $(k+1)-i\not\in L$. Let $b$ be an element in $K$ such that $b^k=\alpha_i$.
Let $g=(g_1,\dots,g_n)\in\ch_0$ be such that $g_{k+1-i}= b^{-(k-1)}$ while every other entry $g_j=b$. Let 
$[L]=[l_1,\dots,l_k]$.
 Then $g\cdot[L]= g_{l_1}\dots g_{l_k}[L]$.
 
Suppose that $k+1-i\in L$. Then one of $g_{l_1},\dots,g_{l_k}$ is equal to $b^{-(k-1)}$ while the other $k-1$ are equal to $b$. Thus, $g_{l_1}\dots g_{l_k}=b^{k-1}b^{-(k-1)}=1$ and so $g\cdot[L]=[L]$. 

Now assume that $k+1-i\not\in L$. Then each of the $g_{l_j}$ is equal to $b$ and so $g\cdot[L]=b^k[L]=\alpha_i[L]$.

This shows that the action of $g$ coincides with the action of $f$, as required.

The final case to consider is $f=(\alpha_0;\alpha_1,\dots,\alpha_k;\beta_1,\dots,\beta_p)\in\ch_1$ where $\alpha_0$ is not equal to $1$ but all other entries are equal to $1$. Thus, $f\cdot[u]=\alpha_0[u]$. 
 Let $[L]$ be any other quantum Pl\"ucker coordinate in $\oqgkn$ with $[L]=[I\mid J][u]$ for some quantum minor $[I\mid J]\in\oqmkp$. Note that $f\cdot[I\mid J]=[I\mid J]$ for all quantum minors $[I\mid J]$. Hence, 
 $f\cdot[L] = (f\cdot[I\mid J])(f\cdot[u]) = [I\mid J].\alpha_0[u]=\alpha_0[L]$. Thus, $f\cdot[L] = \alpha_0[L]$ for 
all quantum minors $[L]$ of $\oqgkn$. Let $b\in K$ be such that $b^k=\alpha_0$ and set $g=(b,\dots,b)$. 
Then $g\cdot[L]=b^k[L]=\alpha_0[L]$ for all $[L]$ and the actions of $f$ and $g$ coincide. 
\end{proof} 

The action of $\ch_1$ on $\oqgkn$ is not faithful, as we will see in the next proposition. Let $\ch$ denote $\ch_1$ factored by this kernel of this action. Then $\ch$ acts faithfully on $\oqgkn$. 
The next result shows that $\ch$ is isomorphic to a torus $(K^*)^n$.

\begin{proposition}\label{proposition-hdash}
 The group $\ch$ is isomorphic to a torus $(K^*)^n$.
\end{proposition}

\begin{proof}
The kernel of the action of $\ch_1$ on $\oqgkn$ is the same as the kernel of the action of $\ch_1$ on 
$\oqgkn[[u]^{-1}]=\oqmkp[y^{\pm 1}]$. Using the right hand side, it is easy to check that this kernel is
$\{(1;\lambda, \dots, \lambda;\lambda^{-1},\dots,\lambda^{-1}) \mid \lambda\in\k^*\}$. Hence, choosing $\lambda=\beta_p$, we see that 
$h=(\alpha_0;\alpha_1,\dots,\alpha_k;\beta_1,\dots,\beta_p)\in\ch_1$ has the same action 
as $h':=(\alpha_0;\alpha_1\beta_p,\dots,\alpha_k\beta_p;\beta_1\beta_p^{-1},\dots,\beta_{p-1}\beta_p^{-1},1)\in\ch_1$. It is also easy to check that two distinct elements in $\ch_1$ that each have $1$ in the final place act differently on $\oqmkp[y^{\pm 1}]$, and so the claim is established. 
\end{proof}

In view of this result, we refer to the actions on $\oqgkn$ provided by $\ch$ as {\em the torus automorphisms of $\oqgkn$}.

In the case that $n=2k$, so that $k=p$, the dehomogenisation equality states that 
\[
\co_q(M(k,k))[y^{\pm 1}; \sigma] = \oqgkn[[u]^{-1}].
\]
In this case,  a simple analysis using the formula $[L]=[I\mid J][u]$ shows that the extra automorphism of $\co_q(M(k,k))$ given by transposition of the $x_{ij}$ variables extends to an automorphism of $\co_q(M(k,k))[y^{\pm 1}; \sigma] = \oqgkn[[u]^{-1}]$ 
which, when restricted to $\oqgkn$, gives rise to the diagram automorphism $\tau$ of Corollary~\ref{corollary-diagram-auto}. We will denote this automorphism by $\tau$ for each of the three algebras $\oqgkn$, $\co_q(M(k,k))$ and 
$\co_q(M(k,k))[y^{\pm 1}; \sigma] = \oqgkn[[u]^{-1}]$.

In the case where $2k=n$, let $(\alpha_0; \alpha_1, \dots , \alpha_k;\beta_1, \dots, \beta_k)\in \ch_1$. 
It is easy to check that 
$$
\tau \circ (\alpha_0; \alpha_1, \dots , \alpha_k;\beta_1, \dots, \beta_k) \circ \tau  = (\alpha_0; \beta_1, \dots, \beta_k;\alpha_1, \dots , \alpha_k)\in \ch_1,
$$
 so that $\left<\tau\right>$ acts on $\ch_1$. Also, if $(1;\lambda, \dots, \lambda;\lambda^{-1},\dots,\lambda^{-1})$ is in the kernel of the action of $\ch_1$ on $\oqgkn[[u]^{-1}]$ then 
\begin{align*}
\tau\circ(1;\lambda, \dots, \lambda;\lambda^{-1},\dots,\lambda^{-1})\circ\tau&=
(1;\lambda^{-1}, \dots, \lambda^{-1};\lambda,\dots,\lambda)\\
&=(1;\lambda^{-1}, \dots,\lambda^{-1};(\lambda^{-1})^{-1},\dots,(\lambda^{-1})^{-1})
\end{align*}
is also in the kernel of this action and so $\left<\tau\right>$ acts on $\ch$. 

\begin{definition}\label{definition-autos}
Set $\ca:=\ch$ when $2k\neq n$ and $\ca:=\ch\rtimes\langle\tau\rangle$ when $2k=n$.
\end{definition}

The analysis above shows that the elements of $\ca$ act naturally as automorphisms of $\oqgkn$ via the dehomogenisation equality. 

\begin{claim}\label{claim-all-autos} 
The automorphism group of $\oqgkn$  is $\ca$. 
\end{claim} 

We will prove this claim in the following sections.


\section{Adjusting automorphisms}
\label{section-adjusting-autos}

The quantum grassmannian $\oqgkn$ carries the structure of an $\mn$-graded algebra generated in degree one when we give degree one to each of the quantum Pl\"ucker coordinates. In addition, \cite[Theorem 5.3]{llr} shows that $\oqgkn$ is a unique factorisation domain in the sense of Chatters \cite{chatters}. According to Chatters, an 
element $p$ of a noetherian domain $R$ is said to be
prime if (i) $pR = Rp$, (ii) $pR$ is a height one prime ideal of $R$, and (iii) $R/pR$ is an
integral domain. A noetherian domain R is then said to be a {\em unique factorisation
domain}  if $R$ has at least one height one prime ideal, and
every height one prime ideal is generated by a prime element.

In this section, we exploit  these properties of $\oqgkn$  in a series of results  to see that given an arbitrary automorphism of $\oqgkn$ we can essentially fix the minimal and maximal elements in the poset after allowing adjustment of the automorphism by elements of $\ch$. 

\begin{lemma}\label{degree-one-normal}
Let $A=\oplus_{i=0}^\infty\, A_i$ be a graded algebra that is a domain with $A_0$ equal to the base field and $A$ generated in degree one. 
Suppose that $a=a_1+\dots+a_m$ is a normal element with $a_i\in A_i$ for each $i$. Then 
$a_1$ is a normal element. 
\end{lemma} 

\begin{proof} If $a_1=0$ then there is nothing to prove; so assume that $a_1\neq 0$.
As $A$ is generated in degree one, it is enough to check normality with respect to homogeneous elements of degree
one; so suppose that $b\in A_1$. Then
$ba = ba_1+\dots+ba_m = ac= (a_1+\dots +a_m)(c_0+c_1+\dots+c_t)$ for some 
$c= c_0+c_1+\dots+c_t\in A$ with $c_i\in A_i$. Comparing degree one terms gives 
$0=a_1c_0$; so $c_0=0$. The degree two terms then show that $ba_1=a_1c_1\in a_1A$, and this demonstrates that $a_1$ is normal. 
\end{proof} 

\begin{lemma} \label{prime-element}
Let $A=\oplus_{i=0}^\infty\, A_i$ be a graded algebra that is a domain with $A_0$ equal to the base field. Suppose also that $A$ is a unique factorisation domain.

Let $a$ be a homogeneous element of degree one that is normal. Then $a$ generates a prime ideal of height one.
\end{lemma} 

\begin{proof}  Let $P$ be a prime that is minimal over the ideal 
$a R$. By the noncommutative principal ideal theorem \cite[Theorem 4.1.11]{mcr}, the height of $P$ is one. Hence, $P=pA$ for some normal element $p$, as $A$ is a UFD. Thus, $a$ is a (right) multiple of $p$. By degree considerations, 
$p$ must have degree one and $a$ must be a scalar multiple of $p$. Thus, $a$ and $p$ 
generate the same ideal, which is the prime ideal $P$. This establishes the claim. 
\end{proof} 


The remaining results in this section all deal with $\oqgkn$. 
As in earlier sections,  let $[u]=[1,\dots,k]$. 
\begin{lemma}\label{d(I)-the-same}
Suppose that $a=\sum a_I[I]\neq 0$, with $a_I\in\k$, is a linear combination of quantum Pl\"ucker coordinates that 
is a normal element. 
Then $d(I)$ is the same for each $I$ that has $a_I\neq 0$. 
\end{lemma} 

\begin{proof}  Since $a=\sum a_I[I]$ has degree one, $a$  is irreducible, as well as being normal. Hence, the ideal 
$P$ generated by $a$ is a height one prime ideal of $\oqgkn$, 
by Lemma~\ref{prime-element}.

If $a=a_I[I]$ for some $[I]$ then the result holds. Thus we may assume that at least two scalars $a_I$ are nonzero. In particular, $aK\neq [u]K$, and so $[u]\not\in a\oqgkn=P$. 

Let $J$ be such that $d(J)$ is as small as possible among those $d(I)$ for which  $a_I\neq 0$. We will show that $d(I)=d(J)$ for all $I$ such that $a_I\neq 0$. 

Now
\[
[u]a= 
[u]\left(\sum_{I\neq J}\, a_I[I] +a_J[J]\right)= \left(\sum_{I\neq J}\, a_Iq^{d(I)}[I] +a_Jq^{d(J)}[J]\right)[u],
\]
by using Lemma~\ref{lemma-how-u-commutes}, 
and so 
\[
\left(\sum_{I\neq J}\, a_Iq^{d(I)}[I] +a_Jq^{d(J)}[J]\right)[u]- q^{d(J)}\left(\sum_{I\neq J}\, a_I[I] +a_J[J]\right)[u]
=[u]a-q^{d(J)}a[u]
\in P;
\]
that is, 
\[
b:=\left(\sum_{I\neq J}\, (q^{d(I)}-q^{d(J)})a_I[I]\right)[u]\in P
\]

As $[u]\not\in P$ this gives 
\[
b=\left(\sum_{I\neq J}\, (q^{d(I)}-q^{d(J)})a_I[I]\right)\in P=a\oqgkn.
\]
Thus, there is a scalar $\lambda\in K$ with $b=\lambda a$. If $\lambda\neq 0$ then this is a contradiction, as 
$[J]$ occurs 
nontrivially in $\lambda a$ and not in $b$. Therefore, $b=0$ and so each $q^{d(I)}-q^{d(J)}=0$. This forces $d(I)=d(J)$, since $q$ is not a root of unity. 
\end{proof} 

In order to prove the next lemma, we need to check that we can apply \cite[Proposition 4.2]{ll1}  to the quantum grassmannian. In order to do this, we need to observe that for each quantum Pl\"ucker coordinate $[L]$ there exists a 
 quantum Pl\"ucker coordinate $[L']$ such that $[L][L'] = \alpha[L'][L]$ for some $1\neq\alpha\in K$. We use Lemma~\ref{lemma-how-u-commutes} to do this. If $[L]=[u]=[1,\dots,k]$ then set $[L']=[1,\dots,k-1,k+1]$ so that  $[L][L'] = q[L'][L]$ while if  $[L]\neq[u]$ set $[L']=[u]$ so that  $[L][L'] = q^{-d(L')}[L'][L]$.

\begin{lemma} \label{sigma-u-one-is-u}

Let $\rho$ be an automorphism of $\oqgkn$. 
Then $\rho([u])_1=\lambda [u]$, for some $\lambda\in K^*$.
\end{lemma}

\begin{proof} 
Note that $\rho([u])$ is a normal element of $\oqgkn$; so 
Lemma~\ref{degree-one-normal} applied to $\rho([u])$ shows that 
$\rho([u])_1$ is a normal element. If $\rho([u])_1=\lambda [u]$, for some $\lambda\in K^*$, 
then we are finished; 
so assume that $\rho([u])_1\neq \lambda [u]$. Then $\rho([u])_1 = \sum a_I[I]$ for some 
$[I]$ with each $d(I)=d>0$, by Lemma~\ref{d(I)-the-same}.

Let $w\in\oqgkn$ be such that $\rho(w)=[u]$, or, equivalently, 
$\rho^{-1}([u])=w$. Note that the degree zero term of $w$ must be zero. 
Write $w=w_1+w_{>1}$. As $\rho(w_{>1})_1=0$, by  \cite[Proposition 4.2]{ll1}, we see that $[u]=\rho(w) = \rho(w_1) + \rho(w_{>1})$; so that 
$[u]=\rho(w)_1=\rho(w_1)_1$. If $w_1=\lambda[u]$ then 
$[u]=\rho(w_1)_1=\rho(\lambda [u])_1=\lambda\rho([u])_1$ and so 
$\rho([u])_1=\lambda^{-1}[u]$, a contradiction; so we have  $w_1\neq \lambda u$, for any $\lambda\in K^*$. By applying 
Lemma~\ref{d(I)-the-same} to $w_1$, we may write
$w_1=\sum b_J[J]$ for some $[J]$ with each $d(J)=e>0$, say. Hence, 
$[u]w_1=q^ew_1[u]$.

Consider the degree two term in $\rho([u]w)$. We know that 
$\rho([u]w)=\rho([u])\rho(w)$ and that the degree zero terms of 
$\rho([u])$ and $\rho(w)$ are both zero, by \cite[Proposition 4.2]{ll1}. Hence, 
\[
\rho([u]w)_2 = \rho([u])_1\rho(w)_1 = \left(\sum a_I[I]\right)[u]= q^{-d}[u]\left(\sum a_I[I]\right).
\]
On the other hand, $\rho([u]w) = \rho([u]w_1+[u]w_{>1})=\rho([u]w_1)+\rho([u]w_{>1})$. Now, 
$[u]w_{>1}$ has no term in degree less than three. Hence, the same is true for 
$\rho([u]w_{>1})$, by \cite[Proposition 4.2]{ll1}, and, in particular, 
$\rho([u]w_{>1})_2=0$.
Thus, $\rho([u]w)_2=\rho([u]w_1)_2$. However, 
$\rho([u]w_1) = \rho(q^ew_1[u]) =q^e\rho(w_1)\rho([u])$.
Therefore, 
\[
\rho([u]w)_2= \rho([u]w_1)_2 
=\rho(q^ew_1[u])_2 =q^e\rho(w_1)_1\rho([u])_1
= q^e[u]\left(\sum a_I[I]\right)
\]

The two expressions we have obtained for $\rho([u]w)_2$ must be equal; 
so $q^{-d}[u](\sum a_I[I])=q^e[u](\sum a_I[I])$. Hence, 
$q^{-d}=q^e$; so that $q^{e+d}=1$. As $q$ is not a root of unity and 
$e+d>0$, this is a contradiction, and our lemma is proved.
\end{proof} 


\begin{lemma} \label{tau-almost-fixes-u}
Let $\rho$ be an automorphism of $\oqgkn$. 
Then $\rho([u])=\lambda [u]$, for some $\lambda\in K^*$. 
\end{lemma} 

\begin{proof} 
The element $\rho([u])$ is a normal element, and the degree zero term of $\rho([u])$ is equal to $0$, by \cite[Proposition 4.2]{ll1}.  
Suppose that the degree of $\rho([u])$ is $t$ and that  $\rho([u]) = a_1+\dots+a_t$ with 
$\deg(a_i)=i$. Recall that $a_1=\lambda [u]$, for some $\lambda\in K^*$, by Lemma~\ref{sigma-u-one-is-u}. 
There is an element $r\in \oqgkn$ with 
$[u]\rho([u])=\rho([u])r$. The degree of $r$ must be one. Assume $r=r_0+r_1$ 
with $r_i$ having degree $i$. Thus, 
\[
[u](\lambda [u]+a_2+\dots+a_t)=(\lambda [u]+a_2+\dots+a_t)(r_0+r_1).
\]
As there is no term in degree one on the left hand side of the above equation, we must have 
$r_0=0$. Looking at terms in degree two, we then see that 
$\lambda [u]^2=\lambda [u]r_1$; so that $r_1=[u]$ and 
$[u]\rho([u])=\rho([u])[u]$.

Write $\rho([u])$ in terms of the standard  basis for $\oqgkn$, as in \cite{klr}, say 
$\rho([u])= \sum \alpha_i S_i$, where $\alpha_i$ is in the field and each $S_i$ is 
a standard monomial. If $S= [I_{1}]\dots[I_{m}]$ is such a standard monomial, 
then set $d(S):= \sum d(I_i)$ and note that each $d(I_i)\geq 0$ with $d(I_i)=0$ if and only 
if $[I_i]=[u]$. Then, $S[u]=q^{d(S)}[u]S$, and so $S[u]=[u]S$ if and only if $d(S)=0$ (in which case 
$S=[u]^m$ for some $m$). 

In any case, note that $[u]S$ is a standard monomial, as $[u]$ is the 
unique minimal quantum Pl\"ucker coordinate. 
Hence, 
\[
\sum \alpha_i [u]S_i =
[u]\rho([u])= \rho([u])[u] = \left(\sum \alpha_i S_i\right)[u] 
= [u] \left(\sum \alpha_i q^{d(S_i)}S_i\right) = \sum \alpha_i q^{d(S_i)}[u]S_i
\]

As the extreme left and right terms in the above display are in the standard basis, this 
forces $d(S_i)=0$ whenever $\alpha_i\neq 0$. Hence, $\rho([u])$ must be a polynomial 
in $[u]$. 

The same argument applies to the automorphism $\rho^{-1}$; so $\rho^{-1}([u])$ 
must also  be a polynomial 
in $[u]$.

Suppose that $\rho([u]) =\sum_{i=1}^t \alpha_i[u]^i$ with $\alpha_t\neq 0$, and, similarly, 
suppose that $\rho^{-1}([u])= \sum_{i=1}^s \beta_i[u]^i$ with $\beta_s\neq 0$

Then, 
\[
[u]=\rho^{-1}\rho([u]) = \rho^{-1}\left(\sum_{i=1}^t \alpha_i[u]^i\right) 
=\sum_{i=1}^t \alpha_i\rho^{-1}([u])^i = \alpha_1\beta_1[u]+\dots+\alpha_t\beta_s [u]^{st}
\]
Therefore, $s=t=1$, and $\rho([u])=\lambda [u]$, for some $\lambda\in K^*$, as required.
\end{proof} 

The above result refers to $[u]=[1,\dots,k]$, the extreme leftmost quantum Pl\"ucker coordinate. 
We want to establish a similar result for $[w]:=[n-k+1,\dots,n]$, the extreme rightmost 
quantum Pl\"ucker coordinate. In order to do this we employ an antiautomorphism of 
$\oqgkn$ which we now describe. 

Let $w_0$ denote the longest element on the symmetric group on $n$ elements; that is, 
$w_0(i)=n+1-i$. 
The discussion immediately before Proposition 2.12 of \cite{ag} shows that the map 
$\theta:\oqgkn\longrightarrow\oqgkn$ given by $\theta([I]) = [w_0(I)]$ for each quantum Pl\"ucker 
coordinate $[I]$ is an antiautomorphism. Note that 
$\theta([u]) =\theta([1,\dots,k])=[n+1-k,\dots, n]=[w]$.

\begin{corollary}\label{adjust-u}%
Let $\rho$ be an automorphism of $\oqgkn$. Then there exists $h\in\ch$ 
such that $(h\circ\rho)([u])=u$ and $(h\circ\rho)([w])=[w]$. 
\end{corollary} 

\begin{proof} 
The map $\theta\rho\theta$ is an automorphism of $\oqgkn$. By Lemma~\ref{tau-almost-fixes-u}, 
there is an element $\mu\in K^*$ such that $\theta\rho\theta([u])=\mu [u]$. 
Apply $\theta$ to both sides of this equality to obtain $\rho\theta([u])=\mu\theta([u])$; 
that is, $\rho([w])=\mu [w]$. We also know that $\rho([u])=\lambda [u]$ for some $\lambda\in K^*$. 
Set $h:=(\lambda^{-1},1,\dots,1,\mu^{-1})\in\ch_0$. Then 
$(h\circ\rho)([u])=[u]$ and $(h\circ\rho)([w])=[w]$. Lemma~\ref{lemma-H0-is-in-H1} shows that the action of an element of $\ch_0$ is realised by the action of an element of $\ch_1$ and hence of $\ch$, so the result follows.
\end{proof}

In what follows, we will often replace the original automorphism $\rho$ by $h\circ\rho$ 
so that we may assume that $\rho([u])=[u]$ and $\rho([w])=[w]$ in calculations. 

\section{Transfer to quantum matrices} \label{section-transfer}

Recall from the discussion in Section~\ref{section-basic-definitions} that when discussing  $\oqgkn$ we are assuming that $1<k$ and that $2k\leq n$. 

Let $\rho$ be an automorphism of $\oqgkn$. Set $[u]=[1\dots k]$ and $[w]=[n-k+1,\dots,n]$. 
By using Corollary~\ref{adjust-u}, at the expense of 
adjusting $\rho$ by an element of $\ch$,  
we can, and will,  assume that $\rho([u])=[u]$ and $\rho([w])=[w]$. The automorphism 
$\rho$ now extends to $\oqgkn[[u]^{-1}]$, and so to $\co_q(M(k,p))[y^{\pm1};\sigma] $, by the dehomogenisation equality of Section~\ref{section-dehomogenisation}, and we know that $\rho(y)=y$.

We will show that such a $\rho$ sends $\co_q(M(k,p))$ to itself. Once we have done 
this, we will know how 
$\rho$ acts on each quantum minor in $\co_q(M(k,p))$ as we know the automorphism group of $\co_q(M(k,p))$. We can then calculate how $\rho$ acts on arbitrary quantum Pl\"ucker coordinates of $\oqgkn$, by using the formulae of Lemma~\ref{lemma-to-and-fro}.

From the discussion in Section~\ref{section-dehomogenisation}, we know that the quantum matrix generators 
$x_{ij}$ are defined by 
\[
x_{ij}:=[1,\dots,\widehat{k+1-i},\dots k, j+k][u]^{-1},
\]
for $1\leq i\leq k$ and $1\leq j\leq p$.
In the following calculations, all quantum minors $[-\mid-]$ are formed from the generators $x_{ij}$ of $\co_q(M(k,p))$. 


As $2k\leq n$, we know that $k\leq n-k=p$.
In this case, the quantum minor 
$[I\mid J]:= [1,\dots, k\mid p+1-k, \dots, p]$ is defined (we are using all the rows 
of $\co_q(M(k,p))$ and the last 
$k$ columns (and there are at least $k$ columns, by the assumption)). 

Now, Lemma~\ref{lemma-to-and-fro} shows that 
$[I\mid J] = [p+1\dots n][1\dots k]^{-1}=[w][u]^{-1}$, and it follows this  that $\rho([I\mid J])=[I\mid J]$ for any element $\rho\in \ch$ such that $\rho([u]) = [u]$ and $\rho([w]) = [w]$.

 We can calculate how 
$[I\mid J]$ commutes with $[u]=[1\dots k]$. Note that $k<n-k+1$, as $2k\leq n$. 
Thus the index sets $\{1,\dots, k\}$ and $\{p+1,\dots, n\}$ do not overlap, and 
\[
[u][I\mid J]= [u][w][u]^{-1} = q^k[w][u][u]^{-1}
= q^k[w][u]^{-1}[u] = q^k[I\mid J]\,[u], 
\]
where the second equality comes from Lemma~\ref{lemma-how-u-commutes}.  

Also, we know how $[I\mid J]$ commutes with the $x_{ij}$ by the following lemma.

\begin{lemma} \label{lemma-commutation-rules}
 (i) If $j\geq n+1-2k=p+1-k$ then 
$x_{ij}[I\mid J]=[I\mid J]\,x_{ij}$.\newline (ii) If $j<n+1-2k=p+1-k$ then 
$x_{ij}[I\mid J]=q[I\mid J]\,x_{ij}$.
\end{lemma}

\begin{proof} 
(i) In this case, $x_{ij}$ is in the quantum matrix algebra determined by the rows from $I$ and the 
columns from $J$ and $[I\mid J]$ is the quantum determinant of this algebra, so the claim follows as $[I\mid J]$ is 
central in this algebra.\\
(ii) This result is obtained from the first equation in E(1.3c) in \cite[Section 1.3]{gl}. \end{proof}

We define two gradings on $T:=
\co_q(M(k,p))[y^{\pm 1}; \sigma] = \oqgkn[[u]^{-1}]$ which grade $T$ according to how elements commute with $y=u$ and $[I\mid J]=[1,\dots, k\mid p+1-k, \dots, p]$.

First, set $T_i:=\{a\in T\mid yay^{-1}=q^ia\}$.

\begin{lemma}
(i)  $T= \bigoplus_{i=1}^\infty\,T_i$,\\
(ii) $\rho(T_i)=T_i$.
\end{lemma}

\begin{proof}
Note that $T$ is generated by $y^{\pm 1}$ and the $x_{ij}$, and that $y^{\pm 1}\in T_0$, while $x_{ij}\in T_1$, as 
$yx_{ij}=qx_{ij}y$. As $\oqmkp$ is an iterated Ore extension with the elements $x_{ij}$ added lexicographically, the elements of the form $x_{11}^{a_{11}}\dots x_{kp}^{a_{kp}}y^s$ with $a_{ij}\geq 0$ and $s\in\mz$ form a basis for $T$. Part (i) now follows as $yx_{ij}y^{-1}=qx_{ij}$. Part (ii) follows from the fact that $\rho(y)=y$.
\end{proof} 

Next, set $T^{(i)}:=\{a\in T\mid [I\mid J]a[I\mid J]^{-1}=q^{-i}a\}$. The commutation rules given in Lemma~\ref{lemma-commutation-rules} show that $x_{ij}\in T^{(0)}\cup T^{(1)}$. Also, note that 
$y[I\mid J] =[u][I\mid J]=q^k[I\mid J][u] =  q^k[I\mid J]y$, so $y=u\in T^{(k)}$
 and $y^{-1}\in T^{(-k)}$.
 
\begin{lemma}
 (i) $T= \bigoplus_{i\in\mz}\,T^{(i)}$,\\
(ii) $\rho(T^{(i)})=T^{(i)}$.
\end{lemma}

\begin{proof}
Part (i) is proved as in the previous lemma, and Part (ii) follows from the fact that $\rho([I\mid J])=[I\mid J]$.\\
\end{proof} 

\begin{lemma}
$\left(T^{(0)}\cup T^{(1)}\right)\cap T_1\subseteq\oqmkp$.
\end{lemma}

\begin{proof}
Suppose that $a\in\left(T^{(0)}\cup T^{(1)}\right)\cap T_1$. Then 
$a$ is a sum of scalar multiples of monomials of the form $m:=x_{11}^{a_{11}}\dots x_{kp}^{a_{kp}}y^s$ with each $a_{ij}\geq 0$ and $s\in\mz$. Such a monomial is in $T_f$, where 
$f=\sum a_{ij}$ and so we must have $\sum a_{ij}=1$, as $a\in T_1$. Thus, only one $x_{kl}$ can occur in each monomial and $a$ is a sum of scalar multiples of monomials of the form $m:=
x_{kl}y^b$. Such an $m$ is in $T^{(e+bk)}$, where $e=0$ or $e=1$. As each $m$ must be in 
$T^{(0)}\cup T^{(1)}$ we must have $e+bk= 0$ or $e+bk=1$. The only possible solutions to these restrictions is that $b=0$, as $k\geq 2$. Hence, $a$ is a sum of scalar multiples of monomials of the form $x_{kl}$ which means that $a\in \oq(M(k,p))$.
\end{proof} 

\begin{theorem}\label{theorem-reduced-auto}
Let $\rho$ be an automorphism of $\oqgkn$ with $1<k$ and $2k\leq n$ that satisfies 
$\rho([u])=[u]$ and $\rho([w])=[w]$. Then 
$\rho$ extends naturally to an automorphism of $T=\oqgkn[[u]^{-1}]=\oq(M(k,p))[y^{\pm 1};\sigma]$ 
such that $\rho(y)=y$ and $\rho(\oq(M(k,p)))=\oq(M(k,p))$.
\end{theorem} 

\begin{proof}
We know that $\rho(y)=y$ and so only need to show that $\rho(\oq(M(k,p)))=\oq(M(k,p))$. To do this, it is sufficient to prove that $\rho(x_{ij})\in \oq(M(k,p))$ for each generator $x_{ij}$.

Note that $x_{ij}\in\left(T^{(0)}\cup T^{(1)}\right)\cap T_1$ for each $i,j$, and so $\rho(x_{ij})\in\left(T^{(0)}\cup T^{(1)}\right)\cap T_1$ for each $i,j$. 

Hence, $\rho(x_{ij})\in\oqmkp$ by the previous lemma.
\end{proof} 



\section{The automorphism group of $\oqgkn$}\label{section-main-result}

In Definition~\ref{definition-autos} and Claim~\ref{claim-all-autos} 
 we  identified a group of automorphisms of $\oqgkn$ that we claimed would give us all the automorphisms. We are now in a position to justify this claim.  
  
 Recall from Definition~\ref{definition-autos} that  $\ca=\ch$ when $2k\neq n$ and $\ca:=\ch\rtimes\langle\tau\rangle$ when $2k=n$.

\begin{theorem}\label{theorem-main} 
The automorphism group ${\rm Aut}(\oqgkn)$ of $\oqgkn$ is isomorphic  to $\ca$.
\end{theorem} 

\begin{proof} 
Let $\rho$ be an arbitrary automorphism of $\oqgkn$. By  Corollary~\ref{adjust-u}, there is  an automorphism $h\in\ch$ such that $h\cdot\rho([u])=[u]$ and $h\cdot\rho([w])=[w]$.  It is enough to prove that this adjusted automorphism is in $\ch$ or in $\ch\rtimes\langle\tau\rangle$; so we may assume that $\rho([u])=[u]$ and $\rho([w])=[w]$. 
This adjusted automorphism satisfies the hypothesis for 
Theorem~\ref{theorem-reduced-auto} and so $\rho$ extends naturally to an automorphism of $T=\oqgkn[[u]^{-1}]=\oq(M(k,p))[y^{\pm 1};\sigma]$ 
such that $\rho(y)=y$ and $\rho(\oq(M(k,p)))=\oq(M(k,p))$. Given this, the action of $\rho$ is completely determined by its restriction to $\oq(M(k,p))$; so it is enough to show that $\rho$ restricted to $\oq(M(k,p))$ is realised by an element of $\ch$ or, if $2k=n$, that   either $\rho$ or $\rho\circ\tau$ is realised by an element of $\ch$.

If $2k \neq n$, then by \cite[Corollary 4.11 and its proof\,]{ll1} $\rho$ is determined on $\oq(M(k,p))$by row and column operations and so is in $\ch$, as required.

If $2k=n$, then by \cite[Theorem 3.2]{y1} either $\rho$ or $\rho\circ\tau$ is determined by row and column operations and so is in $\ch$. In either case, $\rho\in\ch\rtimes\langle\tau\rangle$, as required.
\end{proof}





~\\~\\
\begin{minipage}{\textwidth}
\noindent S Launois\\
School of Mathematics, Statistics and Actuarial Science,\\
University of Kent\\
Canterbury, Kent, CT2 7FS,\\ UK\\[0.5ex]
email: S.Launois@kent.ac.uk \\

\noindent T H Lenagan\\
Maxwell Institute for Mathematical Sciences,\\
School of Mathematics, University of Edinburgh,\\
James Clerk Maxwell Building\\
The King's Buildings\\
Peter Guthrie Tait Road\\
Edinburgh EH9 3FD \\
           UK\\[0.5ex]
email: tom@maths.ed.ac.uk\\
\end{minipage}



\begin{thebibliography}{MM}


\bibitem{ac} J Alev and M Chamarie, {\em Derivations et automorphismes de quelques algebras quantiques}, Comm Alg 20 (1992), 1787 -- 1802.

\bibitem{ag} J M Allman and J E Grabowski, {\em A quantum analogue of the dihedral action on
Grassmannians}, Journal of Algebra 359 (2012), 49 -- 68.

\bibitem{chatters} A W Chatters, {\em Non-commutative unique factorisation domains}, Math. Proc. Cambridge Philos. Soc. 95 (1984), 4 -- 54.


\bibitem{gl} K R Goodearl and T H Lenagan, {\em Primitive ideals in quantum $SL_3$ and $GL_3$}, New Trends in Noncommutative Algebra, Contemporary Mathematics 562 (2012), 115 -- 140.

\bibitem{gw} K R Goodearl and R B Warfield Jr., An Introduction to Noncommutative Noetherian Rings. Second edition. London Mathematical Society Student Texts, 61. Cambridge University Press, Cambridge, 2004. xxiv+344 pp. 

\bibitem{gy} K R Goodearl and M T Yakimov, {\em Unipotent and Nakayama automorphisms of quantum nilpotent algebras}, Commutative algebra and noncommutative algebraic geometry. Vol. II, 181 -- 212, Math. Sci. Res. Inst. Publ., 68, Cambridge Univ. Press, New York, 2015.

\bibitem{klr}  A C Kelly, T H Lenagan and L Rigal,  {\em Ring theoretic properties of quantum grassmannians}, J. Algebra Appl. 3 (2004),  9 -- 30.

\bibitem{kl} D Krob and B Leclerc, {\em Minor identities for Quasi-Determinants and Quantum Determinants}, Commun. Math. Phys. 169 (1995), 1 -- 23.


\bibitem{launois} S Launois, {\em Primitive ideals and automomorphism group of $U_q^+(B_2)$},  J. Algebra Appl. 6 (2007), no. 1, 21 -- 47.

\bibitem{ll1} S Launois and T H Lenagan, {\em Primitive ideals and automorphisms of quantum
matrices}, Algebr. Represent. Theor. 10 (4), (2007), 339 -- 365.


\bibitem{ll2}  S Launois and T H Lenagan, {\em Automorphisms of quantum matrices}, Glasgow Mathematical Journal 55A (2013), 89 -- 100. Special issue in honour of 60th birthdays of Kenny Brown and Toby Stafford.

\bibitem{lln} S Launois, T H Lenagan and B M Nolan, {\em Total positivity is a quantum phenomenon: the grassmannian case}, arXiv:1906.06199, to appear in Memoirs of the American Mathematical Society.

\bibitem{llr}S Launois, T H Lenagan and L Rigal, {\em Quantum unique factorisation domains}, J. London Math. Soc. (2) 74 (2006), no. 2, 321 -- 340.

\bibitem{llr-selecta} S Launois, T H Lenagan and L Rigal,  {\em Prime ideals in the quantum Grassmannian}, Selecta Math. (N.S.) 13 (4) (2008), 697 -- 725.

\bibitem{lr} T H Lenagan and E J Russell, {\em Cyclic orders on the quantum Grassmannian}, Arab. J. Sci. Eng. Sect. C Theme Issues 33 (2) (2008), 337 -- 350.

\bibitem{mcr} J C McConnell and J C Robson, Noncommutative noetherian rings. With the cooperation of L. W. Small. Revised edition. Graduate Studies in Mathematics, 30. American Mathematical Society, Providence, RI, 2001. xx+636 pp.



\bibitem{y1} M Yakimov, {\em The Launois--Lenagan  conjecture}, J Algebra 392 (2013), 1 -- 9.


\bibitem{y2} M Yakimov, {\em Rigidity of quantum tori and the Andruskiewitsch--Dumas conjecture}, Selecta Math. 20 (2014), 421 -- 464.
\end{thebibliography}
\end{document}